\documentclass{article}
\usepackage{stmaryrd}
\usepackage{latexsym}
\usepackage{amssymb}
\usepackage{amscd}
\usepackage{amsopn}
\usepackage{amsfonts}
\usepackage{amsmath}
\usepackage{graphicx}
\usepackage{amsthm}
\usepackage{psfrag}
\usepackage[hidelinks]{hyperref}
\usepackage{mathtools}

\def \cR {\mathbb{R}}

\def \cC {\mathbb{C}}
\def \cN {\mathbb{N}}
\def \cZ {\mathbb{Z}}

\def \Lau {\mathbb{L}}

\def \fp {\mathrm{fp}}
\def \FP {\mathrm{FP}}
\def \mon {\mathrm{e}}

\def \pdeg {\mathrm{pdeg}}
\def \revlex {\mathrm{rlex}}

\def \hom {\mathrm{hom}}

\def \DD {\mathrm{DD}}
\def \ddd {\mathrm{H}}

\def \supp {\mathrm{supp}}

\def \mon {\mathfrak{m}}

\def \reg {\mathrm{reg}}

\newtheorem{Theorem}{Theorem}[section]
\newtheorem*{Conjecture}{Conjecture}

\newtheorem{Proposition}[Theorem]{Proposition}

\newtheorem{Corollary}[Theorem]{Corollary}

\newtheorem{Rem}[Theorem]{Remark}
\newtheorem{Exam}[Theorem]{Example}

\newenvironment{Remark} {\begin{Rem} \rm }{\end{Rem}}
\newenvironment{Example} {\begin{Exam} \rm }{\end{Exam}}
\title{Solutions of the equation $a_n + (a_{n-1} + \cdots (a_2 +  (a_1 + x^{r_1})^{r_2}\cdots )^{r_{n}} = b\, x$}

\newcommand{\Address}{{
		\bigskip
		\footnotesize
		D.~Panazzolo, \textsc{IRIMAS, Université de Haute-Alsace, 68093 Mulhouse, France}\par\nopagebreak
		\textit{E-mail address}: \texttt{daniel.panazzolo@uha.fr}
		
		}}

\newcommand{\tmaffiliation}[1]{\\ \small#1}

\begin{document}
\author{
 Daniel Panazzolo\footnote{This material is based upon work supported by the Agence Nationale de la Recherche under Grant ANR-23-CE40-0028.}
  \tmaffiliation{Université de Haute-Alsace}
 }
\date{}
\maketitle
\begin{flushright}  Dedicated to Jorge Sotomayor\end{flushright} 

\begin{abstract}
We establish a novel upper bound for the real solutions of the equation specified in the title, employing a generalized derivation-division algorithm. As a consequence, we also derive a new set of Chebyshev functions adapted specifically for this problem.
\end{abstract}
\vspace{0.5cm}
\section*{Introduction}
The equation
\begin{equation}\label{equation-principal}
a_n + (a_{n-1} + \cdots (a_2 +  (a_1 + x^{r_1})^{r_2}\cdots )^{r_{n}} = b\, x
\end{equation}
frequently emerges in problems concerning the qualitative theory of vector fields. In these problems, the variables $(b,a,r)$, belonging to $\mathbb{R}_{>0}\times\mathbb{R}^{2n}$, are treated as free parameters. The goal is to estimate the number of real solutions $x$ to this equation, uniformly with respect to the parameters.

For instance, consider a smooth planar vector field $X$ whose phase portrait has a hyperbolic polycycle $\Gamma$ with $n$ saddle points $s_1,\ldots,s_n$ and let $X_\lambda$ be a smooth family of vector fields which unfolds $X$. We fix a local transverse section $\Sigma$ (see picture below) such that $\Sigma\cap \Gamma = \{p\}$ 
and consider the first return map $\pi:(\Sigma,p) \rightarrow (\Sigma,p)$, which depends on the parameter $\lambda$. The limit cycles bifurcating from $\Gamma$ are related to the isolated fixed points of $\pi$.

\begin{figure}[htb]
\begin{center}
\includegraphics[height=5cm]{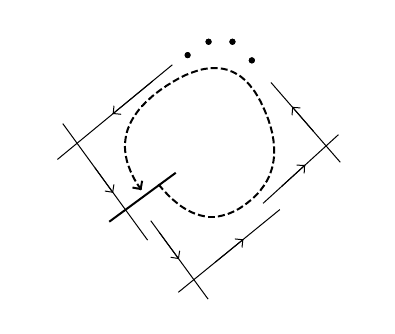}
\label{default}
\end{center}
\end{figure}

According a result of Mourtada (see \cite{Mo}, Theorem 1), in the vicinity of the parameter value $\lambda=0$, there exists an appropriate domain on the transversal where the fixed-point equation $\pi(x)=x$ has an expansion of the form (\ref{equation-principal}), up to some error term which is innocuous under generic conditions. Here the parameters $(a,r)$ depend smoothly on $\lambda$ and measure respectively the distance between the saddle separatrices and the ratio of eigenvalues at the saddle points. The additional parameter $b>0$ is related to the first multiplier of the Poincaré map along the unperturbed polycycle and is only relevant if the product $r_1\cdots r_n$ equals one.

In such particular setting, one is interested in finding an upper bound for the {\em cyclicity of the origin}. Namely, the number of solutions (\ref{equation-principal}) bifurcating from $x=0$ under small perturbation of the parameters.

Another situation where (\ref{equation-principal}) appears is related to slow-fast systems. More specifically, it is shown in \cite{DR} that such equation appears in the study of limit cycles under bifurcation from multi-layer canard cycles. In contrast to the previous situation, here one is interested in finding a {\em global} uniform bound for the number of isolated solutions.  

Independently, the structure of the group $G$ of real functions generated by translations $x \mapsto x+a$ and power maps $x \mapsto x^r$ is a subject of interest in group theory (see e.g.~\cite{Co}). Here, one is interested in the abstract structure of such group and one natural question is whether  $G$ is isomorphic to a free product. From a dynamical systems point of view, this problem is related to the following question, usually called {\em center problem}: Characterise those hyperbolic polycycles which lie on the boundary of a continuum of periodic orbits, i.e.~those polycycles for which the Poincaré first return map $\pi$ is the identity.  We refer the reader to \cite{Pa} for some results on this direction.

In this paper, we will be interested in estimating the maximum number of real solutions of (\ref{equation-principal}). To state the problem precisely, we need some definitions: Given $(a,r) \in \cR^{2n}$, let $I_n(a,r) \subset \cR$ denote the subset of all real numbers $x$ such that the functions $g_0,\ldots,g_{n}$ defined by 
\begin{equation}\label{seqg}
g_0(x) = x\quad \mbox{ and }\quad g_{i+1}(x) = a_{i+1} + g_{i}(x)^{r_{i+1}}, \mbox{ for } i = 1,\ldots,n-1.
\end{equation}
are all strictly positive.  Since all these functions are strictly monotone (increasing or decreasing), it is easy to see that $I_n(a,r)$ is an open (possibly empty) interval of $\cR$.  If the $r_i$ are all positive, $I_n(a,r)$ is a nonempty neighborhood of infinity.
\begin{Remark}
The set $\{(a,r,I_n(a,r)) : (a,r) \in \cR^n \times \cR^n\}$ is an open subset of $\cR^{2n+1}$ which belongs to the o-minimal structure of sets definable in the expansion of the real field by the Pfaffian functions (in the sense of Khovanskii).
\end{Remark}
So, our goal is to estimate the number of connected components of the  {\em solution set }
\begin{equation}\label{maineq}
S_n(b,a,r) = \{x \in I_n(a,r) \mid a_n + (a_{n-1} + \cdots (a_1 + x^{r_1})^{r_2}\cdots )^{r_{n}} - b x = 0\},
\end{equation}
Let $\FP_n(b,a,r)$ denote the number of such connected components. It follows immediately from the analyticity of the functions in (\ref{seqg}) that for each fixed parameter $(b,a,r)$, either $S_n(b,a,r) \equiv I_n(a,r)$ or $S_n(b,a,r)$ is a finite set of points.  As a consequence, $\FP_n(b,a,r) < \infty$, and we define 
\begin{equation}\label{definition-fp}
\fp(n)=\sup_{(b,a,r)} \FP_n(b,a,r)
\end{equation}
The first obvious question is whether $\fp(n) < \infty$. This is a direct Corollary of Khovanskii's Fewnomial theory \cite{K}.  Indeed, by introducing auxiliary variables 
$(x_1,\ldots,x_n) \in \cR^n_+$, 
equation (\ref{equation-principal}) can be equivalently written as a  system,
\begin{eqnarray*}
a_1 + x_1^{r_1} = x_2 \hfill\\
a_2 + x_2^{r_2} = x_3 \hfill\\
\vdots\\
a_{n} + x_n^{r_{n}} = b\, x_1 \hfill
\end{eqnarray*}
In such form, we have a system of $n$ equations in $n$-variables with $2n+1$ distinct monomials, and the theory of Khovanskii provides the bound
$$
\fp(n) \le \mathrm{K}(n) := 2^{n(2n-1)}(n+1)^{2n}
$$
(see \cite{K}). In  \cite{BS}, Sotille and Bihan slightly improved the fewnomial's bound. Their result implies that $\fp(n) \le \mathrm{SB}(n) := \frac{e^2+3}{4}\, 2^{\binom{n}{2}} n^n$.

One of the difficulties of the problem is that, in general, the sequence
(\ref{seqg}) is not a {\em Chebyshev system} (see e.g.~\cite{Cop}, Section 3.3, for the definition). To the author's knowledge, this was first observed in Mourtada's thesis \cite{JMM}. In that work, Mourtada constructed a specific example of a generic hyperbolic polycycle with four singularities which bifurcates into five limit cycles.  In our notation, this result translates to $\fp(4) \geq 5$. 
We shall see in Section~\ref{sect-isolatedroots} that $\fp(4) \le 13$ and, in Section~\ref{sec-casen5}, that in fact $\fp(3)=5$. The exact value of $\fp(n)$ is unknown for $n\ge 4$.

In the next Sections, we shall describe a simple algorithm which gives a new upper bound for $\fp(n)$ by a derivation-division algorithm. The algorithm is somewhat inspired in \cite{LRW}. 

Such new
upper bound coincides with 
the exact value of $\fp(n)$ for $n \in \{1,2,3\}$ and is significantly smaller the Khovanskii's and Bihan-Sottile's bound for $n \le 5$. However, it is 
largely outperformed by such bounds for $n \ge 6$. We refer to Section \ref{sect-numberofsteps} for a more detailed comparison.   

One of the advantages of our method is that it provides as a by-product a system of linearly independent monomials which allows to expand the left-hand side of 
(\ref{equation-principal}) uniformly with respect to the parameters $(a,r)$. This expansion can be interpreted as a generalization of the {\em compensator-type expansion} introduced independently by Roussarie \cite{Ro} and Ecalle \cite{Ec}. We refer to Section \ref{sect-compensator} for the details.

\section{Derivation/division for Laurent polynomials}
We initially consider a more abstract setting of a certain ring equipped with a derivation.

Given $n\in\cN$, let $R=\cZ[R_1,..,R_n]$ denote the commutative ring of polynomials $n$ variables $R_1,\ldots,R_n$ with integer coefficients. Our basic object will be the ring 
$$\Lau_n= R[x_1^{\pm},y_1^\pm,\ldots,x_n^{\pm},y_n^\pm],$$
i.e.~the ring of Laurent polynomials in $x_i,y_i$ with coefficients in $R$. We will frequently use the fraction notation and write, for instance, $x_1 y_1^{-1}$ simply as $x_1/y_1$. We convention that $\Lau_0 = R$.

A {\em monomial} in $\Lau_n$ is polynomial of the form $\mon=c\, x^k y^l$, with a nonzero coefficient $c\in R\setminus \{0\}$ and where we note
$$
x^k y^l := x_1^{k_1} \cdots x_n^{k_n} y_1^{l_1} \cdots y_n^{l_n}
$$
The vector $(k,l) \in \cZ^{n}\times \cZ^{n}$ will be called the {\em exponent} $\mon$ and we say  $\mon$ {\em unitary} if $c = 1$. The {\em (poly)-degree} of a such monomial is the integer vector $\pdeg(\mon) = k+l \in \cZ^{n}$. We note that all unitary monomials are invertible element of $\Lau_n$.

Each non-zero polynomial $p \in \Lau_n$ can be written as a finite sum of monomials
\begin{equation}
p = \sum_{i \in I} \mon_i
\end{equation}
where the set $\supp(p) = \{\mon_i : i \in I\}$ will be called the {\em support} of $p$.  
We now define successively three subrings of $\Lau_n$.
\begin{equation}\label{three-subrings}
\Lau_n^\reg\subset \Lau_{n-1}^\hom[x_n,y_n]\subset \Lau_n^\hom \subset \Lau_n
\end{equation}
Firstly, we denote by $\Lau_n^\hom \subset \Lau_n$ the subring of Laurent polynomials which are {\em homogeneous}, i.e.~such that all monomials in their support have a same degree. We will note by $\pdeg(p) \in \cZ^{n}$ the degree of an element $p$ of $\Lau_n^\hom$. For each $1\le j \le n$, we will denote by $\pdeg_j(p) \in \cZ$ the degree of $p$ with respect to the variables $(x_j,y_j)$ (i.e.~the $j^{th}$-component of $\pdeg(p)$).

We now define $\Lau_{n-1}^\hom[x_n,y_n] \subset \Lau_n^\hom$ as the subring of homogeneous Laurent polynomials which are {\em polynomials} in the $(x_n,y_n)$-variables.  In other words, we consider the Laurent polynomials $p$ such that each monomial $\mon=c x^k y^l \in \supp(p)$ has an exponent $(k,l)$ satisfying $k_n,l_n \ge 0$. Thus, if we let $m = \pdeg_n(p)$ and denote the variables $(x_n,y_n)$ simply as $(x,y)$, we can expand such element $p$ in the form
\begin{equation}\label{expansion-p}
p = \sum_{j=0}^{m} q_j x^j y^{m-j}
\end{equation}
with coefficients $q_j$ in $\Lau_{n-1}^\hom$.

Finally, we define the smallest subring $\Lau_n^\reg \subset \Lau_{n-1}^\hom[x_n,y_n]$ appearing in (\ref{three-subrings}), whose elements will be called {\em $n$-regular polynomials}. The definition is by induction on $n$:
\begin{itemize}
\item[(i)]For $n = 0$, the 0-regular polynomials are simply $L_0^\reg = R$.
\item[(ii)]For $n \ge 1$, we say that a Laurent polynomial $p \in \Lau_{n-1}^\hom[x_n,y_n]$ is {\em $n$-regular} if the coefficient $q_0$ in the expansion (\ref{expansion-p}) is $(n-1)$-regular.
\end{itemize}
Note that, in particular, if $p\in \Lau_n^\reg$ then $\pdeg(p) \in \cN^n$ is a positive vector and if $\pdeg_n(p) = 0$ then $p$ is $(n-1)$-regular. Inductively, we show that $\pdeg(p)=0$ if and only if $p$ belongs to the base ring $R$. 

Let us give a different characterisation of $\Lau_n^\reg$. Consider a homogeneous Laurent polynomial $p\in \Lau_n^\hom$ of degree $\pdeg(p)=(m_1,..,m_n) \in \cN^n$. Then $p$ is $n$-regular if and only if there exists a sequence of nonzero Laurent polynomials $p_j \in \Lau_{j-1}^\hom[x_j,y_j]$ for $0\le j\le n$ such that:
\begin{itemize}
\item[(1)] $p_n = p$ and
\item[(2)] for each $1\le j\le n$, $\pdeg_j(p_j) = m_j$ and we can inductively write
$$
p_j = Q_j + p_{j-1} y_j^{m_j}
$$
where $Q_j  \in \Lau_{j-1}^\hom[x_j,y_j]$ is divisible by $x_j$ (as a polynomial in $(x_j,y_j)$).  In other words, $Q_j$ contains only monomials of the form $\star x_j^\alpha y_j^\beta$ with positive exponents $\alpha,\beta$ such that $\alpha+\beta = m_j$ and $\beta < m_j$. 
\end{itemize}
This result leads to the following nested form for $p$,
\begin{equation}\label{nested-expansion-p}
p = Q_n + (\cdots (Q_2 + (Q_1 + \alpha y_1^{m_1})y_2^{m_2})\cdots)y_n^{m_n}
\end{equation}
for some $\alpha \in R$.
\begin{Example}\label{examples-reg}
For $n=2$, the Laurent polynomials
$$
p = \frac{x_1}{y_1} x_2 + y_2\qquad q=\left( \frac{x_1^2}{y_1}+\frac{x_1^3}{y_1^2}+x_1+\frac{y_1^2}{x_1} \right) x_2 y_2 + (x_1+y_1)y_2^2 
$$
are regular, of respective degrees $\pdeg(p) = (0,1)$ and $\pdeg(q) = (1,2)$
\end{Example}
We now turn $\Lau_n$ into a differential ring by introducing a derivation 
$$
\partial:\Lau_n \rightarrow \Lau_n
$$ 
which satisfies $\partial(R)=0$ and acts on the variables $x_j,y_j$ according to the following formulas
\begin{equation}\label{fp-der}
\partial(y_1) = \partial(x_1) = \Delta_1 x_1,\qquad \partial(y_j) = \partial(x_j) = \Delta_j x_j\, \frac{x_1 \cdots x_{j-1}}{y_1\cdots y_{j-1}},\qquad j\ge 2
\end{equation}
where we define $\Delta_j = R_1\cdots R_j \in R$. The following result is an obvious consequence of the above expressions for $\partial$ and the Leibniz rule. 
\begin{Proposition}\label{prop-derivation-keeps-pdeg}
{\rm (1)} The derivation $\partial$ maps the ring $\Lau_{j-1}^\hom[x_j,y_j]$ into itself, for each index $1\le j \le n$.  \\
{\rm (2)} Given $p \in \Lau_{j-1}^\hom[x_j,y_j]$, we either have $\partial p = 0$ or else
$$
\pdeg(\partial p) = \pdeg(p)
$$
in other words, $\partial$ preserves the degree.
\end{Proposition}
\begin{proof}
By linearity, it suffices to consider the case of a monomial $\mon = x^k y^l$. Then, from the Leibniz rule, have
$$
\partial \, \mon  = x^k y^l \sum_{j=1}^{n} k_j \frac{\partial(x_j)}{x_j} + l_i \frac{\partial(y_j)}{y_j} = x^k y^l \sum_{j=1}^{n} \Delta_j P_{j-1}(k_j +  l_j \frac{x_j}{y_j})
$$
where $P_{j-1}= \frac{x_1\cdots x_{j-1}}{y_1 \cdots y_{j-1}}$ and $P_0 = 1$.  

We now remark that each term $P_{j-1}$ is a Laurent homogeneous polynomial of degree $0$. Therefore, by the additivity of the degree under multiplication, we easily obtain
$$
\pdeg( \partial\, \mon) = \pdeg(\mon).
$$
Finally, if we further suppose that $m$ belongs to $\Lau_{n-1}^\hom[x_n,y_n]$ (i.e.~
$k_n, l_n \ge 0$) then it follows from the above formula that $\partial \mon$ is a sum of monomials lying in $\Lau_{n-1}^\hom[x_n,y_n]$. The same argument works if $\mon$ lies in $\Lau_{j-1}^\hom[x_j,y_j]$.
\end{proof}
We observe however that the ring $\Lau_n^\reg$ {\em is not} preserved by $\partial$, as the following simple example shows:
\begin{Example} \label{example-derivationnonreg}
Consider the polynomial $p$ given in Example \ref{examples-reg}. Then, 
$$
\partial p = \left(\Delta_1 \frac{x_1}{y_1} + (2 \Delta_2-\Delta_1) \frac{x_1^2}{y_1^2} \right) x_2
$$
which is not a Laurent regular polynomial according to our definition.
\end{Example}
In order to deal with such phenomena, we observe that any element of $\Lau^\hom_n$ can be transformed into a regular Laurent polynomial upon division by an appropriate monomial.
 
More precisely, the {\em regularization monomial} associated to a non-zero Laurent polynomial $p \in \Lau^\hom_n$ is a unitary monomial $\mon = x^k y^l$ inductively defined as follows:
\begin{itemize}
\item[(i)] Case $n = 0$: We define $\mon=1$.
\item[(ii)] Case $n \ge 1$: If we note $m=\pdeg_n(p)$ then we observe that we can write an expansion of $p$ in the variables $(x,y) = (x_n,y_n)$ as 
$$
p = \sum_{j=n_0}^{n_1} q_j x^j y^{m-j}
$$
for some (possibly negative) integers $n_0 \le n_1$ and coefficients $q_j \in \Lau_{n-1}^\hom$ such that $q_{n_0},q_{n_1}$ are non-zero. We then define
$\mon = x^{n_0} y^{d-n_1} \bar \mon$, where $\bar{\mon}$ is the regularization monomial of $q_{n_0}$. 
\end{itemize}
We will denote such monomial $\reg(p) = \mon$.  The following result is obvious:
\begin{Proposition}
The Laurent polynomial $p/\reg(p)$ is regular. 
\end{Proposition}
\begin{Example}
Consider the polynomial $q=\partial p$ given in Example \ref{example-derivationnonreg}. Then, the above definitions gives $\reg(q) = \frac{x_2}{x_1}{y_1^2}$. The Laurent polynomial
$$
r \coloneq \frac{q}{\reg(q)} =  \Delta_1 y_1 + (2 \Delta_2-\Delta_1) x_1 + \Delta_1 y_1 
$$
is regular, and has degree $\pdeg(r) = (1,0)$. 
\end{Example}
\begin{Theorem}
Suppose that $p\in \Lau_n^\reg$. Then either $\partial p = 0$ or the regular polynomial  
$$
q = \frac{\partial p}{\reg(\partial p)}
$$
is such that $\pdeg(q) <_\revlex \pdeg(p)$.
\end{Theorem}
\begin{Remark}
Here we denote by $<_{\revlex}$ the usual reverse lexicographical ordering in $\cN^n$, namely
$$
(m_1,..m_n) <_{\revlex} (m_1',..,m_n') \Longleftrightarrow \exists i : (m_j < m_j^\prime,i <j\le n) \land m_i < m_i^\prime
$$
\end{Remark}
\begin{proof}
If $\pdeg (p) = 0$ then necessarily $p \in R$ and hence $\partial p = 0$. So, let us assume that $\pdeg(p) = (m_1, \ldots, m_n)$ is such
that $m_1 = \cdots = m_{j - 1} = 0$ and $m_j \geqslant 1$, for some $1 \le j \le n$. We claim that the polynomial $q$ defined in the enunciate
satisfies
\begin{equation}\label{pdegqlex}
 \pdeg (q) \le_{\revlex} (0, \ldots, 0, m_j - 1, m_{j +1}, \ldots, m_n)
\end{equation}
Indeed, it follows from the definition of regularity and the assumption on $\pdeg(p)$ that the expansion (\ref{nested-expansion-p}) can be rewritten as
\begin{equation}\label{nested-expansion-p2}
p = Q_n + (\cdots (Q_{j + 1} + (Q_j + \alpha y_j^{m_j}) y^{m_{j + 1}}_{j +
   1}) \cdots) y^{m_n}_n
\end{equation}
where $\alpha \in R \setminus \{ 0 \}$ and, for each index $j \le i \le n$, the coefficient $Q_i$ is an element $\mathbb{L}_{i - 1}[x_i, y_i]^{\hom}$ containing only exponents in
$x_i^{m_i} y_i^0, \ldots, x_i y_i^{m_i - 1}$. We recall that the derivation
$\partial$ applied to each monomial $y_i^{m_i}$ gives
$$
\partial y_i^{m_i} = m_i \Delta_i x_i y_i^{m_i - 1} P_{i - 1} 
$$
where $P_{i - 1} = \frac{x_1 \cdots x_{i - 1}}{y_1 \ldots y_{i - 1}} \in
\mathbb{L}_{i - 1}^{\hom}$.  In particular, for the innermost term in the
expansion (\ref{nested-expansion-p2}), we can write
$$
 \partial (Q_j + \alpha y_j^{m_j}) = \partial Q_j + \alpha \partial
   y_j^{m_j} = \widetilde{P_j} 
$$
where $\tilde{P}_j$ is a Laurent polynomial in $\mathbb{L}_{j - 1} [ x_j,
y_j]^{\hom}$ containing only exponents in $x_j^{m_j} y_j^0, \ldots,
x_j y_j^{m_j - 1}$ (and hence is divisible by $x_j$).

More generally, applying the derivation $\partial$ to the nested expression of $p$ given
above, we obtain
$$
\partial p = \tilde{Q}_n + (\cdots (\tilde{Q}_{j + 1} + \tilde{P}_j y^{m_{j
   + 1}}_{j + 1}) \cdots) y^{m_n}_n
$$
where, for each index $j + 1 \leqslant i \leqslant n$, the coefficient 
$\tilde{Q}_i$ is a Laurent polynomial in $\mathbb{L}_{i - 1} [ x_i, y_i]^{\hom}$
which contains only exponents in $x_i^{m_i} y_i^0, \ldots, x_i y_i^{m_i - 1}$.

Considering separately the cases $P_j = 0$ and $\tilde{P}_j \neq 0$, it now suffices to observe
that the division of $\partial p$ by its regularization monomial $
\reg(\partial p)$ results into a regular polynomial $q$ such that (\ref{pdegqlex}) holds.
\end{proof}
Applying successively the above Theorem, we obtain the following:
\begin{Corollary}[Derivation-division Algorithm]\label{der-div-algorithm}
Let $p \in \Lau_n^\reg$. Then the sequence $(p^{(n)})_{n\ge 0}$ of Laurent polynomials in $\Lau_n^\reg$, defined inductively as
$$
p^{(0)} = p,\qquad p^{(n+1)} = \frac{\partial\, p^{(n)}}{\reg (\partial\, p^{(n)})},\quad n \ge 0
$$
eventually gives an element $p^{(n)} \in R$.
\end{Corollary}
\begin{proof}
It suffices to observe that if $p\in \Lau_n^\reg$ then $\pdeg(p) \in \cN^n$ and moreover $p \in R$ if and only if $\pdeg(p)=0$.
\end{proof}
In the above setting, we define the {\em derivation-division complexity} of $p$ as the smallest integer $m = \DD(p)\ge 0$ such that $p^{(m)} \in R$. 
\begin{Remark}
The basic idea behind the derivation-division algorithm is quite simple: The expressions given in (\ref{fp-der}) show that $\partial(y_j)$ only depends on the variables $x_1,\ldots,x_j$ and $y_1,..,y_{j-1}$. Hence, we inductively eliminate all monomials in $y_1,..,y_n$ and then divide out by the maximal possible factors in $x_1,..,x_n$. 
\end{Remark}
\begin{Example}\label{example-n=1}
Let $n=1$. Writing the variable $(x_1,y_1)$ simply as $(x,y)$, a polynomial $p\in \Lau^\reg$ of degree $m$ has the form
$$
p = \sum_{j=0}^m \alpha_j x^j y^{m-j} 
$$
where the coefficients $\alpha_j$ lie on the ring $R$. 
Applying the derivation $\partial$ to each monomial $x^j y^{m-j}$ gives 
$$
\partial x^j y^{m-j} = \Delta_1 \left( j x^j y^{m-j} + (m-j) x^{j+1} y^{m-j-1} \right)
$$
Therefore, we can write $\partial p = \sum_{j=0}^{m} \beta_j x^j y^{m-j}$ where
$$
\beta_j = \Delta_1 \left( j \alpha_j + (m-j-1) \alpha_{j-1} \right), \quad j\ge 1
$$
and $\beta_0 = 0$. As a consequence, $\partial p$ is divisible by $x$ and the polynomial $p^{(1)} = \partial p/x$ has degree at most $m-1$. 
We conclude by induction that $\DD(p)\le m$.
\end{Example}
\begin{Example}
Let $n=2$. Denoting the variables $(x_1,y_1)$ and $(x_2,y_2)$ simply by $(x,y)$ and $(x',y')$, we consider the special case where $p\in \Lau^{\reg}$ is such that
$$
p \in R[x,x^\prime,y,y^\prime]
$$
In other words, we assume that our initial Laurent polynomial is indeed an homogeneous polynomial. Writing $\pdeg(p) = (m',m)$, we claim that in this case $\DD(p) \le m(m'+1) + m'$. 

If $m=0$, we fall in the situation of the Example \ref{example-n=1} and the estimate holds. So, let us suppose that $m\ge 1$ and write the expansion
$$
p = Q_1 + p_0 y^m
$$
where $Q_1 = \sum_{j=1}^{m}{ q_j x^j y^{m-j}}$ and $p_0$ is a an element of $\Lau^{\reg}_1$.  In fact, due to our hypothesis, $p_0,q_1,..,q_m$ are all homogeneous polynomials of degree $m'$ in variables $(x',y')$.  

After at most $m'$ steps of the derivation-division algorithm, we obtain a new Laurent polynomial of the form
$$
q = R_1 + \alpha y^m,
$$
of degree $\pdeg(q) = (0,m)$.  Here $\alpha \in R$ and we have $R_1 = \sum_{j=1}^{m}{ r_j x^j y^{m-j}}$, where each coefficient $q_j$ is a Laurent polynomial
with support contained in the set of monomials $x'^i y'^{-i}$, with $i\in\{0,..,m\}$. In the subsequent derivation-division step, we will get $r = \partial q/ m$, where $\mon$ is the monomial
$$
m = \frac{x'}{y'^{m'+1}}
$$
and such that, $r$ has degree $(m',m-1)$. One can verify that such new polynomial $r$ satisfies the same hypothesis of $p$ (i.e.~lies in $R[x,x',y,y']$). Therefore, we conclude by induction hypothesis that $\DD(r) \le (m-1)(m'+1) + m'$. Hence, since $\DD(p) \le m' + 1 + \DD(r) = m(m'+1) + m'$.
\end{Example}
\section{Studying the number of isolated roots}\label{sect-isolatedroots}
We are now ready to estimate the the real solutions of equation (\ref{equation-principal}). 

\noindent For each fixed parameter value $(a,r) \in \cR^{2n}$, consider the sequence of functions $g_0,f_1,g_1,..,f_n,g_n$ defined inductively as $g_0(x) = x$, 
$$
f_1(x) = x^{r_1},\qquad g_1(x) = a_1 + f_1(x)
$$ and, for each $1 \le  i \le n-1$,
$$
f_{i+1}(x) = g_i(x)^{r_{i+1}},\qquad g_{i+1}(x)=f_{i+1}(x) + a_{i+1}
$$
As we have seen in the introduction, the functions $f_i,g_i$ are strictly positive and analytic on some open (possibly unbounded) interval 
$I(a,r) \subset \cR$.  We therefore can consider the subring $G_n$ of $C^\omega(I(a,r))$ given by
\begin{equation}\label{ring-Gn}
G_n = r[f_1^\pm,g_1^\pm,\ldots,f_n^{\pm},g_n^{\pm}]
\end{equation}
i.e.~the ring formed by finite linear combinations of monomials $f^k g^l$ (with possible negative exponents $k,l$) and coefficients in the ring $r=\cZ[r_1,..,r_n]$. 

The following result relates this latter ring with the ring of Laurent polynomials defined in the previous Section. 
\begin{Proposition}\label{prop-morphismphi}
The derivation $x \frac{d}{d x}$ maps the ring $G_n$ into itself. Moreover, if we consider the morphism of rings $\Phi:G_n \rightarrow \Lau_n$ defined by $\Phi(r_i) = R_i$ and
$$
\Phi(f_i) = x_i,\qquad  \Phi(g_i) = y_i,\qquad 1\le i \le n,
$$
Then, we have the relation $\partial \circ \Phi = \Phi \circ x \frac{d}{d x}$. 
\end{Proposition}
\begin{proof}
The proof is immediate. Indeed, we have
$$
x \frac{d}{d x}  f_1 = x \frac{d}{d x}  x^{r_1} = r_1 x^{r_1} = r_1 f_1
$$
Similarly,  
$$
x \frac{d}{d x}  f_2 = x \frac{d}{d x}  (a_2 +  (a_1 + x^{r_1})^{r_2} = r_1 r_2 (a_2 +  (a_1 + x^{r_1})^{r_2-1} x^{r_1} = (r_1 r_2) f_2 \left( \frac{f_1}{g_1}\right)
$$
And, in general,
$$
x \frac{d}{d x}  g_j = x \frac{d}{d x}  f_j =   (r_1\cdots r_j)\, f_j\, \frac{f_1 \cdots f_{j-1}}{f_1\cdots f_{j-1}}
$$
Therefore, comparing with the expressions given in (\ref{fp-der}), we immediately conclude that  $\partial \circ \Phi = \Phi \circ x \frac{d}{d x}$.
\end{proof}
Consider now the  analytic function $\varphi: I(a,r) \rightarrow \cR$ given by
\begin{equation}\label{function-varphi}
\varphi(x) =  a_n + (a_{n-1} + \cdots (a_1 + x^{r_1})^{r_2}\cdots )^{r_{n}} - b x
\end{equation}
which we can also write as $\varphi = g_n - b\, g_0$. 

An easy computation shows that $\psi = x^2 \left(\frac{d}{d x}\right)^2 \varphi$ belongs to $G_n$. More specifically, if we consider its
image $q = \Phi(\psi)$ under the morphism $\Phi$ given above and the monomial
$$
\mon = \frac{x_1 \cdots x_n}{(y_1 \cdots y_{n-1})^2}
$$
Then, $p = q/m$ is assumes the very simple form 
\begin{equation}\label{formofp}
p = \Delta_n \sum_{j=1}^{n} (\Delta_{j}-\Delta_{j-1})\; \left(\prod_{i_1=1}^{j-1} x_{i_1} \right)\left(\prod_{i_2=j}^{n-1} y_{i_2} \right) 
\end{equation}
where we define $\Delta_0=1$ and recall that $\Delta_i = R_1\cdots R_i$ for $1\le i\le n$.   

\vspace{0.2cm}
\noindent{\bf Notation: }We denote by $\DD(n) \coloneq \DD(p)$ the number of steps in the derivation-division algorithm for the above polynomial $p$.
\vspace{0.2cm}

\noindent Using the fact that $p \in \Lau_{n-1}^\reg$, we can easily obtain the following result:
\begin{Theorem}\label{Theorem-estimatefp}
The number of isolated roots of $\varphi$ on $I(a,r)$ (counted with multiplicities) is bounded by $\DD(n) + 2$. 
\end{Theorem} 
\begin{proof}
By Rolle's Theorem, the number of roots of $\varphi$ on $I(a,r)$ is bounded by the number of roots of $\psi$ plus two. Let us assume that $\psi$ is not identically zero. Then, the derivation-division algorithm described in Corollary \ref{der-div-algorithm}, when to the polynomial applied $p = \frac{\Phi(\psi)}{m}$
given by (\ref{formofp}), defines a sequence of analytic functions on $I(a,r)$ which eventually (after at most $\DD(n)$ steps) ends up into a constant function. Therefore, applying again Rolle's theorem, we conclude that the number of roots of $\varphi$ on $I(a,r)$ is bounded by $\DD(n)+2$. 
\end{proof}
Since the number $\DD(n)$ does not depend on any specific choice of parameters $(a,r)$, we conclude that:
\begin{Corollary}
The number $\fp(n)$ {\em (see (\ref{definition-fp}))} is bounded by $\DD(n) + 2$. 
\end{Corollary}
\section{At most three solutions for $n = 2$}
As a simple illustration of the algorithm, let us prove that the equation 
\begin{equation}\label{eq2}
a_2 +  (a_1 + x^{r_1})^{r_2} - b\, x = 0
\end{equation}
has a maximum of three isolated solutions. If we denote by $\varphi$ the right-hand side of the above equation, the function $\psi = x^2 \left(\frac{d}{d x}\right)^2 \varphi$ has the form
$$
\psi = r_1 r_2 \left( r_1 \left( {r_2}  - 1 \right) (a_1 + x^{r_1})^{r_2 - 2} x^{2
r_1 - 2} + (r_1 - 1) (a_1 + x^{r_1})^{r_2 - 1} x^{r_1 - 2} \right)
$$
Therefore, upon division by $x^{r_1} (a+x^{r_1})^{r_2-2}$ we obtain
$$
r_1 r_2 \left( r_1 (r_2-1)  x^{r_1} + (r_1-1) (a_1 + x^{r_1}) \right)
$$
which, under the morphism $\Phi$, gives the regular Laurent polynomial
$$p = A_1 x_1 + B_1 y_1$$
where we set $A_1 = \Delta_2 (\Delta_2 - \Delta_1)$ and $B_1 = \Delta_2(\Delta_1 - \Delta_0)$.  
Now, applying one step of the derivation-division, we obtain
$$
A_1 x_1 + B_1 y_1 \overset{\partial}{\Longrightarrow}  \Delta_1(A_1+B_1) x_1 \overset{\cdot/x_1}{\Longrightarrow} \Delta_1(A_1+B_1) \in R
$$
Therefore, $\DD(n) = 1$ and we conclude that $\fp(2) \le 3$. 

The following explicit example shows that indeed $\fp(2) = 3$.  Take $a_1 = 0.004259259259$, 
$a_2 = -0.1516666667$, $r_1 = 2$ and $r_2 = 1/3$.  Numerical computations with SageMath shows that the equation
$$
 -0.1516666667+(0.004259259259+x^2)^{1/3}-x = 0
$$
has the three isolated solutions $x_1 = 0.0123409...$, $x_2 = 0.1741525...$ and
$x_3 = 0.3585065...$.
\section{At most five solutions for $n = 3$}\label{sec-casen5}
Let us prove that the equation 
\begin{equation}\label{eq3}
a_3 + (a_2 +  (a_1 + x^{r_1})^{r_2})^{r_3} - b\, x = 0
\end{equation}
has a maximum of five solutions, for all parameter values $(b,(a,r)) \in \cR_{>0} \times\cR^6$. 

A similar computation to the previous example leads to the polynomial
$$
p = A_2 x_1 x_2 + B_2 x_1 y_2  +
C_2 y_1  y_2
$$
where $A_2 = \Delta_3 (\Delta_3 - \Delta_2)$, $B_2 = \Delta_3 (\Delta_2 - \Delta_1)$ and $C_2 = \Delta_3 (\Delta_1 - \Delta_0)$. 

The following lines show the derivation-division algorithm ends in at most three steps. We omit the dependence on the coefficients $\Delta_i$ to simplify the notation:
$$
 p^{(0)} = x_1 x_2 + x_1 y_2 + y_1 y_2 
$$
$$
\partial p^{(0)} = \left( x_1 + \frac{x_1^2}{y_1} \right) x_2 + x_1 y_2
   \Longrightarrow p^{(1)} = \frac{\partial p^{(0)}}{x_1} = \left( 1 + \frac{x_1 }{y_1}
   \right) x_2 + y_2 
$$
$$
\partial p^{(1)} = \left( \left( \frac{x_1}{y_1} \right)^2 +
   \frac{x_1}{y_1} \right) x_2 \Longrightarrow p^{(2)} = \frac{\partial p^{(1)}}
   {(x_1 / y_1^2) x_2} = x_1 + y_1
$$
$$
 \partial p^{(2)} = x_1 \Longrightarrow p^{(3)} = \frac{\partial p^{(2)}}{x_1} = 1
$$
Therefore, we conclude that $\fp(3)\le 5$. 

The following explicit example shows that indeed $\fp(3) = 5$. We take
$$
a_3 = -8.39 \times 10^{-6},\;\; a_2 = 0.0035836,\;\; a_1 = -0.012\;\; r_1 = \frac{53}{150},\;\; r_2 = \frac{11}{8}, \;\; r_3 = 2
$$
Numerical computations with 20 digits of precision in SageMath give the following five solutions
$$
\begin{array}{l}
x_1 = 0.1270599... \times 10^{-4}\\
x_2 = 0.1921586... \times 10^{-4}\\
x_3 = 0.4764392... \times 10^{-4},\\
x_4 = 0.7949546... \times 10^{-4}\\
x_5 = 0.2384109...
\end{array}
$$
\section{Disconjugacy and compensator-type expansions} \label{sect-compensator}
We now show that, as a byproduct of the derivation-division algorithm, the function appearing in (\ref{equation-principal}) can be expanded in an uniform basis of functions satisfying the Chebyshev property.

Firstly, we recall some basic definitions about disconjugate differential operators. For a thoughtful treatment, we refer the reader to the excellent Coppel's book \cite{Cop}. 
A $m^{th}$-order linear differential operator $L$ defined on an interval $I$ is called 
{\em disconjugate} if it can be written in the form
\begin{equation}\label{disconjugate-op}
L u = \frac{1}{\rho_{m}}\frac{d}{d x} \frac{1}{\rho_{m-1}} \frac{d}{d x} \cdots \frac{d}{d x} \frac{1}{\rho_0} u
\end{equation}
where, for each $0 \le k \le m$,  $\rho_k$ is a strictly positive analytic function on $I$. A well-known result of Polya (see e.g. \cite{Cop}, Chapter 3) states that $L$ is disconjugate if and only if no solution $u$ of $L u = 0$ has more than $m-1$ zeros in $I$, counted with multiplicities. 

Assuming that $L$ is written as in (\ref{disconjugate-op}), we can easily construct a basis of solutions $\{\omega_0,\ldots,\omega_{m-1}\}$ for $L$ as follows: fix an arbitrary base point $x_0 \in I$ and set 
$$
\omega_0(x) =  \rho_0(x),\qquad \omega_1(x) = \rho_0(x) \int_{x_0}^x \rho_1(x_1) d x_1
$$
and, in general, define $\omega_k$ by the $k^{th}$-iterated integral
\begin{equation}\label{basisomega_k}
\omega_k(x) = \rho_0(x) \int_{x_0}^x \rho_1(x_1) \int_{x_0}^{x_1} \rho_2(x_2)\cdots \int_{x_0}^{x_{k-1}} \rho_k(x_k)\, d x_k\,  \cdots\, d x_2\,  d x_1
\end{equation}
Notice that this collection of functions has the following properties:  
\begin{itemize}
\item[(i)] They form a {\em Chebyshev} system, i.e.\ any non-trivial $\cR$-linear combination $\sum \lambda_k \omega_k(x)$ 
has at most $m-1$ zeros on $I$.
\item[(ii)] For each $0\le k \le m-1$, the subset $\{\omega_0,\ldots,\omega_{k-1}\}$ is a 
basis of solutions of the {\em truncated operator}
$$
L^{[k]} u = \frac{1}{\rho_{k}}\frac{d}{d x} \frac{1}{\rho_{k-1}} \frac{d}{d x} \cdots \frac{d}{d x} \frac{1}{\rho_0} u
$$
(with the convention that $L^{[0]} = {1}/{\rho_0}$)
\item[(iii)] $\omega_k(x)$ has a zero of multiplicity precisely $k$ at $x_0$ and $L^{[k]} \omega_k = 1$.
\end{itemize}
The last two properties imply that any solution $u$ of $L u = 0$ can be expanded as
\begin{equation}\label{computation-lambda}
u(x) = \sum_{k=0}^{m-1} \lambda_k\, \omega_k(x)
\end{equation}
where the coefficients $\lambda_k \in \cR$ are determined inductively as follows
$$
\lambda_{m-1} = L^{[m-1]} u,\qquad \lambda_{k-1} = L^{[k-1]}\, \big(u - \sum_{j=k}^{m-1} \lambda_k u_k \big), \quad 1\le k\le m-1
$$ 
Let us now apply these results to our original problem. 

\vspace{0.5cm}
\noindent For each $n \in \cN$ and each parameter value $(b,a,r) \in \cR_{>0}\times\cR^{2n}$, we consider the function $\varphi_{a,r} : I(a,r) \rightarrow \cR$ given by (\ref{function-varphi}). Since the parameter $b$ will be innocuous in the following discussion, we suppose from now on that $b=1$ and will omit it to simplify the notation. 
\begin{Proposition}\label{proposition-conjugate-operator}
The function $\varphi_{a,r}$ is a solution of a disconjugate linear differential operator $L_{a,r}$ of order $\DD(n)+3$.  
\end{Proposition}
\begin{proof}
We recall from the proof of Theorem \ref{Theorem-estimatefp} and Corollary \ref{der-div-algorithm} that if we set $\psi = x^2 \left(\frac{d}{d x}\right)^2$ and consider the Laurent polynomial $q = \Phi(\psi)$ then there exists a sequence of monomials $(\mon_k)$ for $0 \le k \le r-1$,  such that
\begin{equation}\label{operatorinLaurent}
\partial \frac{1}{\mon_{s-1}} \partial \cdots \partial \frac{1}{\mon_0} q = 0
\end{equation}
where $s = \DD(n)$ and $\partial$ is the derivation on $\Lau_n$ defined in (\ref{fp-der}). Using the morphism $\Phi$ defined in Proposition \ref{prop-morphismphi}, this relation gives
$L_{a,r} \varphi_{a,r} = 0$, where $L_{a,r}$ is the disconjugate operator of order $m = s+3$ given by
$$
L_{a,r} = \frac{d}{d x} \left(\frac{x}{h_{r-1}}\right) \frac{d}{d x} \cdots  \left(\frac{x}{h_0}\right) \frac{d}{d x}x^2 \left(\frac{d}{d x} \right)
$$
In this expression, for each $0 \le j \le r-1$, we choose $h_j = \Phi^{-1}(\mon_j)$ to be a function in ring $G_n$ (given in (\ref{ring-Gn})) which is mapped to $\mon_j$ under $\Phi$. 
\end{proof}
In particular, we observe that this result allows to expand $\varphi_{a,r}$ in terms of a {\em globally} defined collection of analytic Chebyshev functions. 
To enunciate this result precisely, we consider the open subset in the parameter space given by
$$
U = \{(a,r) \in \cR^{2n} \mid I(a,r)\text{ is nonempty }\}
$$
Note that, as we have observed in the introduction, $U$ contains in particular the set of all parameters $(a,r)$ such that $r \in \cR^n_{>0}$. Now, we define an open set $V$ in the total space $\{(a,r,x)\in \cR^{2n}\times \cR\}$ by
$$
V = \{\big(a,r,I(a,r)\big) \mid (a,r) \in U\}
$$
\begin{Proposition}[Compensator type expansion for $\varphi_{a,r}$] 
There exists a collection of $m = \DD(n) + 3$ analytic functions $\{\omega_k\}_{0 \le k\le m}$ defined on $V$ such that
\begin{itemize}
\item[(i)] For each fixed parameter value $(a,r)$ in $U$, the functions $\{x \mapsto \omega_{k}(a,r,x)\}$ forms a Chebyshev system on $I(a,r)$.
\item[(ii)] We can write the expansion
$$
\varphi_{a,r}(x) = \sum_{k=0}^{m-1} \lambda_k(a,r) \, \omega_{k}(a,r,x)
$$
where the coefficients $\lambda_k(a,r)$ are analytic functions on $U$. 
\end{itemize}
\end{Proposition}
\begin{proof}
The $m^{th}$-order differential operator $L_{a,r}$ defined in Proposition \ref{disconjugate-op} obviously depends analytically on the parameters $(a,r) \in U$. In other words, the functions
$\rho_0,..,\rho_m$ which appear when we decompose $L_{a,r}$ as (\ref{disconjugate-op}), are analytic on $V$.  

Let $\{\omega_k\}_{0 \le k\le m}$ be the basis of solutions of $L_{a,r}$ defined according to (\ref{basisomega_k}). Note that the base
point $x_0$ used in computation of the iterated integrals must chosen as a point on the interval $I(a,r)$. Such base point, seen as a function of the parameters $(a,r)$, can also be chosen to be analytic.

Finally, when we write $\varphi_{a,r}$ in terms of such basis of solutions $\{\omega_k\}$, the coefficients $\lambda_k(a,r)$ can be computed according to the recursive procedure described in (\ref{computation-lambda}).  This shows that they are analytic functions of $(a,r)$. 
\end{proof}
We will say that ideal ${\cal I}_{a,r} \subset C^\omega(U)$ generated by the functions $\{\lambda_k\}_{0\le k \le m}$ is the {\em ideal of coefficients} of $\varphi_{a,r}$. 
\begin{Remark}
Although the functions $\lambda_k(a,r)$ obviously depend on the choice of the basis $\{\omega_k\}$, the ideal ${\cal{I}}_{a,r}$ itself is independent of this choice.  In fact one 
can prove that ${\cal{I}}_{a,r}$ is indeed a {\em polynomial ideal}, i.e.~generated by polynomials in $\cZ[a,r]$. Let us sketch the proof assuming that $r \in \cR_{>0}$: We consider the Laurent polynomial $q = \Phi(\psi) \in \Lau_n$ defined in the proof of Proposition \ref{proposition-conjugate-operator}. By definition, we can write a finite  expansion
$$
q = \sum_j{c_j\, \mon_j}
$$
where $\mon_j$ are Laurent monomials and $c_j$ are coefficients in $R=\cZ[R_1,..,R_n]$. Under the morphism $\Phi$, each monomial $\mon_j$ corresponds to an analytic function $h_j = \Phi^{-1}(m_j)$ defined on some open interval containing $+\infty$ on the boundary. Therefore, we can consider the asymptotic expansion of $h_j$ at $x=+\infty$. In fact, a simple application of the binomial expansion to each monomial shows that the asymptotic expansion of $q$ has the form
$$
\sum_{\alpha \in S} P_\alpha(r,a) x^\alpha
$$
where $S$ is the semigroup $\cN r_n + \cN r_{n-1}r_n + \cdots + \cN r_1\cdots r_n$ and each $P_j(r,a)$ is a polynomial in $\cZ[a,r]$. Now, it is not very hard to see that ${\cal{I}}_{a,r}$ is generated by $\langle P_\alpha(r,a) \rangle_{\alpha \in S}$. 
\end{Remark}
\begin{Example}[Classical compensators]
Let $n=1$ and consider the function
$$
\varphi = a + x^{r} - x
$$
with $r >0$ and $a \in \cR$.  This function lies in the kernel of the third order disconjugate linear differential operator 
$$
L = \frac{d}{d x} \frac{1}{x^r} \left(\frac{d}{d x}\right)^2
$$
Hence, $\rho_0 = \rho_1 = 1$ and $\rho_2 = x^{r-2}$. Let us compute the basis of solutions $\{\omega_0,\omega_1,\omega_2\}$ by 
integrating from the base point $x_0 = 1$.  We have
$\omega_0(x) = 1$, $\omega_1(x)=x-1$ and 
$$
\omega_2(x) = \int_1^{x} \Omega(r,y) d y, \quad \text{ where }\quad
\Omega(r,y) = 
     \begin{cases}
     \log y, & \mbox{ if }r = 1, \vspace{0.2cm} \\ 
     \displaystyle{\frac{y^{r-1}-1}{r-1}}, & \mbox{ otherwise}
     \end{cases}
$$
is the well-known Ecalle-Roussarie compensator. Explicitly, we obtain 
$$
\omega_2(x) = \begin{cases}
     x \log\left(x\right) - x + 1, & \mbox{ if }r = 1, \vspace{0.2cm} \\ 
     \displaystyle{-\frac{r x - r - x^{r} + 1}{{\left(r - 1\right)} r}}, & \mbox{ otherwise}
     \end{cases}
$$
The ideal of coefficients of ${\cal{I}}_{a,r}$ is generated by the polynomials $\lambda_0=a$, $\lambda_1 = r-1$ and $\lambda_2 = r(r-1)$.
\end{Example}
Motivated by Roussarie's result for the saddle loop case \cite{Ro}, we expect that the generalized compensators introduced in the last Proposition will ultimately prove useful in analyzing the first return map of non-generic hyperbolic polycycles.
\section{The number of steps in the D-D algorithm}  \label{sect-numberofsteps}
We now give an upper estimate on the number $\DD(p)$ of steps of the derivation-division algorithm, when applied on a given Laurent polynomial $p\in \Lau^\reg_n$. Such bound depends solely on the support $\supp(p)$, seen as a subset of $\cZ^{2n}$. 

More precisely, denoting by $\pdeg(p) = (m_1,\ldots,m_n) \in \cN^n$ the degree of $p$, let us suppose that 
\begin{equation}\label{supp-p}
\supp(p) \subset [0,m_1] \times \cdots \times [0,m_n]
\end{equation}
i.e.~that $p$ is a homogeneous polynomial in the variables $x_i,y_i$ (with no negative exponents). For each $1 \le k \le n$, we denote by 
$$\ddd_k : \cN^k \rightarrow \cN$$
the collection of functions defined recursively as follows: 
\begin{itemize}
\item[(1)] If $k = 1$ then $\ddd_1(m) = m$
\item[(2)] If $k \ge 2$ then
\begin{itemize}
  \item[(2.a)] If $m_k = 0$ then $\ddd_k(m_1,\ldots,m_{k-1},0) = \ddd_{k-1}(m_1,\ldots,m_{k-1})$
  \item[(2.b)] If $m_k \ge 1$ then $\ddd_k(m_1,\ldots,m_{k-1},m_k) = \ddd_k(m_1,\ldots,M_{k-1},m_k - 1)$, where $M_{k-1} = m_{k-1}+\ddd(m_1,\ldots,m_{k-1})$.
\end{itemize}  
\end{itemize}
We can now state the following result, whose proof we will omit:
\begin{Proposition}
If $\supp(p)$ satisfies the condition (\ref{supp-p}) then 
$$\DD(p) \le \ddd_n(m_1,\cdots,m_n).$$
\end{Proposition}
Suppose now that $p$ has the special form given by (\ref{formofp}). In this situation, the support of $p$ is quite special and we conjecture, based on an algorithmic implementation, that the number $\DD(n)$ is given by
\begin{Conjecture}
$\DD(n) = A(n-1,1)$.
\end{Conjecture}
In the enunciate, $A(n,k)$ denotes the {\em Ackerman function},  defined inductively as follows:
$$
A(n, k) =
 \begin{cases}
 k+1 & \mbox{if } n = 0 \\
 A(n-1, 1) & \mbox{if } n > 0 \mbox{ and } k = 0 \\
 A(n-1, A(n, k-1)) & \mbox{if } n > 0 \mbox{ and } k > 0.
 \end{cases}
 $$
Notice that such function grows extremely fast with respect to $n$. Here are the first 5 values:
$$
A(0,1) = 2, A(1,1) = 3, A(2,1) = 5, A(3,1) = 13, A(4,1) = 65533, A(5,1) = {2^2}^{{\cdot}^{{\cdot}^{{\cdot}^2}}} - 3
$$
Here, the rightmost expression is a tower of powers of height $A(4,1)\mbox{ + 3}$. 

We conclude by providing a summary of the various upper bounds and the exact values known for $\fp(n)$:
\begin{center} 
\begin{tabular}{|c|cccc|}
\hline
$n$       &    $\DD(n)$  &  $\mathrm{K}(n)$              &  $\mathrm{SB}(n)$ & Exact Value of $\fp(n)$ \\
\hline
1          &         2           &             8                           &          3                       & 2         \\
2          &         3           &             5184                     &         21                      & 3         \\
3          &         5           & $\sim 1.3\times 10^8$      &         562                   & 5         \\
4         &         13          & $\sim 1.0\times 10^{14}$  &         42554               & ?                 \\
5         &         65533    & $\sim 2.0\times 10^{21}$  &      $\sim 8.3 \times 10^{6}$  &?         \\
\hline
\end{tabular}
\end{center}
For $n\ge 6$, the $\DD(n)$ bound becomes much larger than both Khovanskii's bound $\mathrm{K}(n)$ and Bihan Sotile's bound $\mathrm{SB}(n)$. 

\vspace{0.5cm}
\noindent Up to the author's knowledge, the exact value of $\fp(n)$ is unknown for $n\ge 4$.
\section*{Declaration}
The author states that there is no conflict of interest.

\Address


\begin{thebibliography}{100}
\bibitem {BS} Bihan, F.; Sottile, F.{\em New fewnomial upper bounds from Gale dual polynomial system}.Moscow Mathematical Journal, 2007, Volume 7, Number 3, Pages 387–407.
\bibitem {Co}Cohen, Stephen D.{\em The group of translations and positive rational powers is free.}
Quart. J. Math. Oxford Ser. (2) 46 (1995), no. 181, 21--93. 
\bibitem {Cop}Coppel, W. A. (1971). Disconjugacy. In Lecture Notes in Mathematics. Springer Berlin Heidelberg. https://doi.org/10.1007/bfb0058618 
\bibitem{DR} Dumortier F.; Roussarie R. {\em 
Multi-layer canard cycles and translated power functions},
Journal of Differential Equations,
Volume 244, Issue 6,
2008,
Pages 1329-1358.
\bibitem{Ec} Ecalle, J. {\em Compensation of small denominators and ramified
linearisation of local objects}, Ast\' erisque, tome 222 (1994), p. 135-199.
\bibitem {JMM}  Jacquemard, A.;  Khechichine-Mourtada F.;  Mourtada, A. 
{\em Algorithmes formels appliqués à l'étude de la cyclicité d'un polycycle algébrique générique à quatre sommets hyperboliques}
Nonlinearity, Volume 10, Number 1, pp. 19-53.
\bibitem {K} Khovanskii, A.G.{\em Fewnomials}, Trans. of Math. Monographs, 88, AMS, 1991.
\bibitem {LRW}Li, T.; Rojas, J.M.; Wang, X. {\em
Counting real connected components of trinomial curve intersections and $m$-nomial hypersurfaces.} 
Discrete Comput.\ Geom.\ 30 (2003), no.\ 3, 379--414. 
\bibitem {Mo} Mourtada, A. {\em Cyclicit\' e finie des polycycles hyperboliques de champs de vecteurs du plan. Algorithme de finitude.} Annales de l'Institut Fourier, Volume 41 (1991) no. 3, pp. 719-753.
\bibitem {PaUn}Pauer F.; Unterkircher A. {\em
Gr\" obner Bases for Ideals in Laurent Polynomial
Rings and their Application to Systems of Difference
Equations.} Applicable Algebra in Engineering, Communication and Computing 9 (1999), 271--291.
\bibitem{Pa} Panazzolo, D. {\em $\mathrm{PSL}(2,\cC)$, the exponential and some new free groups}, The Quarterly Journal of Mathematics, Volume 69, Issue 1, March 2018, 75–117.
\bibitem{Ro} Roussarie, R. {\em On the number of limit cycles which appear by perturbation of separatrix loop of planar vector fields}. Bol. Soc. Bras. Mat 17, 67–101 (1986). 
\bibitem {Za} Zampieri, Sandro {\em A solution of the Cauchy problem for multidimensional discrete linear shift-invariant systems.}
Linear Algebra Appl.\ 202 (1994), 143--162.
\end{thebibliography}
\end{document}